\documentclass[10pt,twoside,leqno]{siamltex}
\usepackage{amsmath,amsfonts,amssymb}
\usepackage{graphicx,psfrag}
\usepackage{hyperref}
\usepackage{verbatim,enumerate} 
\usepackage{mathrsfs}
\usepackage{color}
\usepackage{ stmaryrd }
\usepackage[table]{xcolor}

\newcommand\tr{\mathrm{tr}}

\newcommand\sgn{\operatorname{sgn}}
\newcommand\Eta{\mathrm{H}}

\allowdisplaybreaks

\begin{document}

\title{Intersection Graphs of Oriented Hypergraphs and Their Matrices}
\author{Nathan Reff\thanks{Department of Mathematics, The College at Brockport{:} State University of New York,
Brockport, NY 14420, USA (\href{mailto:nreff@brockport.edu}{nreff@brockport.edu}).}}

%\subjclass[2010]{Primary 05C50; Secondary 05C65, 05C22}
%\keywords{Oriented hypergraph, intersection graph, line graph, hypergraph adjacency matrix, signed graph, signed hypergraph}

\date{\today}
\maketitle
%%%%%%%%%%%%%%%%%%%%
\begin{abstract}
For a given hypergraph, an orientation can be assigned to the vertex-edge incidences.  This orientation is used to define the adjacency and Laplacian matrices.  In addition to studying these matrices, several related structures are investigated including the incidence dual, the intersection graph (line graph), and the 2-section.  A connection is then made between oriented hypergraphs and balanced incomplete block designs.
\end{abstract}



%%%%%%%%%%%%%%%%%%%%
\section{Introduction}
%%%%%%%%%%%%%%%%%%%%

An {\it oriented hypergraph} is a hypergraph where each vertex-edge incidence is given a label of $+1$ or $-1$.  Shi called this type of hypergraph a {\it signed hypergraph} and used it to model the constrained via minimization (CVM) problem or two-layer routings \cite{Shi1992533,MR1666019}.  Oriented hypergraphs were independently developed to generalize oriented signed graphs \cite{MR1120422} and related matroid properties \cite{ReffRusnak1,MR3118956}.  A generalization of directed graphs, known as {\it directed hypergraphs}, also have this type of vertex-edge labeling (see for example \cite{MR1217096}, and the references therein).  What distinguishes oriented hypergraphs from these other related incidence structures is the notion of an {\it adjacency signature} that naturally allows the adjacency and Laplacian matrices to be defined and studied \cite{ReffRusnak1,MR3218780,Chen2015442}.  This is an alternative approach to studying matrices and hypermatrices associated to hypergraphs \cite{MR1235565,MR1405722,MR1325271,MR2842309,MR2900714,MR3270186,nikiforov}, that does not require a uniformity condition on the edge sizes and allows reasonably quick spectral calculations.  Rodr\'{i}guez also developed a version of the adjacency and Laplacian matrices for hypergraphs without a uniformity requirement on edge sizes \cite{MR1890984}.  The definition of adjacency signature and the derived matrices could be applied to directed hypergraphs and their many applications.

This paper is a continuation of the investigation of matrices and eigenvalues associated to oriented hypergraphs and their related structures.  Specifically, a variety of intersection graphs of oriented hypergraphs are defined and algebraic relationships are found. 

In Section \ref{BackgroundSection} relevant background is provided.  In Section \ref{IntersectionGraphsSection} new oriented hypergraphs are defined, including the intersection graph (or line graph) of an oriented hypergraph.  Some results on oriented hypergraphs that have particular signed graphs as their intersection graphs are shown.  Sections \ref{MatricesofIGSection} and \ref{SwitchingSec} develop  matrix and other algebraic relationships between an oriented hypergraph and its dual and intersection graphs.  These matrix identities are then used to study the eigenvalues associated to the adjacency and Laplacian matrices of the same incidence structures.  In Section \ref{BIBDSec} a connection between oriented hypergraphs, balanced incomplete block designs (BIBD) and their incidence matrices is given.

In \cite{LineDiHyp} the line graph of a directed hypergraph (called a line dihypergraph) is studied.  This construction extends the notion of the line digraph \cite{MR0130839} to the hypergraphic setting in order to study connectivity problems with applications to bus networks.  The definition of the line graph of an oriented hypergraph that we will introduce in this paper is quite different than the line dihypergraph.  However, it would be an interesting future project to see connections between these constructions.

The reader may be interested in two other related investigations.  A generalization of the line digraph, called the partial line digraph, is defined in \cite{MR1178183}.  This too was generalized to directed hypergraphs \cite{MR1889911,MR1881264} with connectivity and expandability results.  Acharya studied signed intersection graphs \cite{MR2791608}, where an alternative hypergraphic version of signed graphs is introduced.

%%%%%%%%%%%%%%%%%%%%
\section{Background}\label{BackgroundSection}
%%%%%%%%%%%%%%%%%%%%

%==================================================================
%==================================================================
\subsection{Oriented Hypergraphs}
%==================================================================
%==================================================================
Let $\Omega$ be a finite set of indexes.  A \emph{hypergraph} is a triple $H=(V,E,\mathcal{I})$, where $V$ is a set, $E=(e_{\iota})_{\iota\in \Omega}$ is a family of subsets of $V$, and $\mathcal{I}$ is a multisubset of $V\times E$ such that if $(v,e_{\iota})\in\mathcal{I}$, then $v\in e_{\iota}$.  Note that an edge may be empty.  The set $V$ is called the \emph{set of vertices}, and $E$ is called the \emph{family of edges}.  We may also write $V(H)$, $E(H)$ and $\mathcal{I}(H)$ if necessary.  Let $n:=|V|$ and $m:=|E|$.  If $(v,e)\in \mathcal{I}$, then $v$ and $e$ are \emph{incident}.  An \emph{incidence} is a pair $(v,e)$, where $v$ and $e$ are incident.  If $(v_i,e)$ and $(v_j,e)$ both belong to $\mathcal{I}$, then $v_i$ and $v_j$ are \emph{adjacent} vertices via the edge $e$.     

A hypergraph is \emph{simple} if for every edge $e$, and for every vertex $v\in e$, $v$ and $e$ are incident exactly once.  Unless otherwise stated, all hypergraphs in this paper are assumed to be simple.  A hypergraph is \emph{linear} if for every pair $e,f\in E$, $|e\cap f|\leq 1$.

The \emph{degree} of a vertex $v_i$, denoted by $d_i=\deg(v_i)$, is equal to the number of incidences containing $v_i$.  The \emph{maximum degree} is $\Delta(H)=\Delta:=\max_i d_i$.  The \emph{size} of an edge $e$ is the number of incidences containing $e$.   A $k$\emph{-edge} is an edge of size $k$.  A \emph{$k$-uniform hypergraph} is a hypergraph such that all of its edges have size $k$.  The {\it rank} of $H$, denoted by $r(H)$, is the maximum edge size in $H$.

The \emph{incidence dual} (or \emph{dual}) of a hypergraph $H=(V,E,\mathcal{I})$, denoted by $H^*$, is the hypergraph $(E,V,\mathcal{I}^*)$, where $\mathcal{I}^*:=\{(e,v):(v,e)\in\mathcal{I}\}$.  Thus, the incidence dual reverses the roles of the vertices and edges in a hypergraph.  

The set of size 2 subsets of a set $S$ is denoted by $\binom{S}{2}$.  The set of \emph{adjacencies} $\mathcal{A}$ of $H$ is defined as $\mathcal{A}:=\{(e,\{v_i,v_j\})\in E\times \binom{V}{2}: (v_i,e)\in \mathcal{I}\text{ and }(v_j,e)\in \mathcal{I}\}$.  Observe that if $\{v_i,v_j\}\in \binom{V}{2}$, then the vertices $v_i$ and $v_j$ must be distinct.  The set of \emph{coadjacencies} $\mathcal{A}^*$ of $H$ is defined as $\mathcal{A}^*:=
\mathcal{A}(H^*)$.

An \emph{oriented hypergraph} is a pair $G=(H,\sigma)$ consisting of an \emph{underlying hypergraph} $H=(V,E,\mathcal{I})$, and an \emph{incidence orientation} $\sigma \colon\mathcal{I}\rightarrow\{+1,-1\}$.  Every oriented hypergraph has an associated \emph{adjacency signature} $\sgn\colon\mathcal{A}\rightarrow \{+1,-1\}$ defined by
\begin{equation}
\sgn(e,\{v_i,v_j\})=-\sigma (v_i,e)\sigma (v_j,e).
\end{equation}
Thus, $\sgn(e,\{v_i,v_j\})$ is called the \emph{sign} of the adjacency $(e,\{v_i,v_j\})$.  Instead of writing $\sgn(e,\{v_i,v_j\})$, the alternative notation $\sgn_{e}(v_i,v_j)$ will be used.  The notation $\sigma_G$ may also be used for the incidence orientation when necessary.  See Figure \ref{OHEx} for an example of an oriented hypergraph.

\begin{figure}[!ht]
    \includegraphics[scale=0.7]{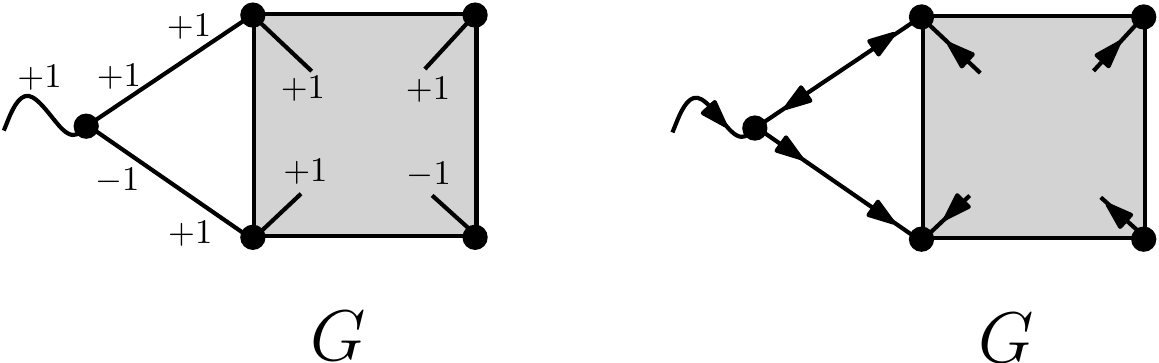}\centering
    \caption{A simple oriented hypergraph $G$ drawn in two ways.  On the left, the incidences are labeled with $\sigma$ values.  On the right, the $\sigma$ values assigned to the incidences are drawn using the arrow convention of $+1$ as an arrow going into a vertex and $-1$ as an arrow departing a vertex.}\label{OHEx}
\end{figure}

As with hypergraphs, an oriented hypergraph has an incidence dual.  The \emph{incidence dual} of an oriented hypergraph $G=(H,\sigma)$ is the oriented hypergraph $G^*=(H^*,\sigma^*)$, where the \emph{coincidence orientation} $\sigma ^{\ast }\colon\mathcal{I}^{\ast }\rightarrow \{+1,-1\}$ is defined by $\sigma ^{\ast }(e,v)\allowbreak=\sigma
(v,e)$, and the \emph{coadjacency signature} $\sgn^*\colon\mathcal{A}^*\rightarrow \{+1,-1\}$ is defined by
\begin{equation*}
\sgn^*(v,\{e_i,e_j\})=-\sigma^*(e_i,v)\sigma^*(e_j,v)=-\sigma (v,e_i)\sigma (v,e_j).
\end{equation*}
The notation $\sigma_{G^*}$ may also be used for the coincidence orientation.  See Figure \ref{OHDUAL} for an example of the dual.

\begin{figure}[!ht]
    \includegraphics[scale=0.7]{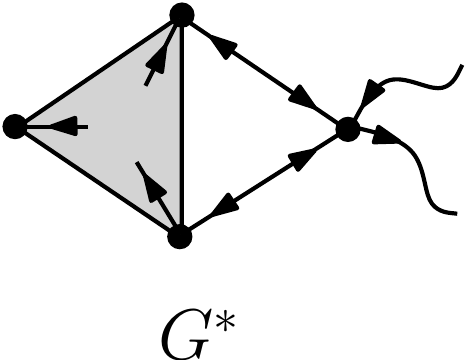}\centering
    \caption{The incidence dual $G^*$ of $G$ from Figure \ref{OHEx}.}\label{OHDUAL}
\end{figure}

%==================================================================
%==================================================================
\subsection{Oriented Signed Graphs}
%==================================================================
%==================================================================

A {\it signed graph} is a graph where edges are given labels of either $+1$ or $-1$.  Formally, a signed graph $\Sigma=(\Gamma,\sgn)$ is a graph $\Gamma$ together with a {\it sign function} (or {\it signature}) $\sgn\colon E(\Gamma)\rightarrow \{+1,-1\}$.  An {\it oriented signed graph} $(\Sigma,\beta)$ is a signed graph together with an {\it orientation} $\beta\colon \mathcal{I}(\Gamma)\rightarrow \{+1,-1\}$ that is consistent with the sign function via the relation 
\begin{equation}\label{OSGrelation}
\sgn(e_{ij})=-\beta(v_i,e_{ij})\beta(v_j,e_{ij}).
\end{equation}
Oriented signed graphs were developed \cite{MR1120422} to generalize Greene's bijection between acyclic orientations and regions of an associated hyperplane arrangement \cite{zbMATH03604927}.  This generalization is intimately connected with the theory of oriented matroids \cite{MR1744046}.

Notice that Equation \ref{OSGrelation} is the same formula used to calculate the adjacency signs in an oriented hypergraph.  This is merely because oriented signed graphs are the model for the generalization of oriented hypergraphs.  It should therefore be no surprise that a 2-uniform oriented hypergraph is an oriented signed graph.  If we say that $G$ is a simple oriented signed graph, then $G$ has no loops or multiple edges.  Put another way, $G$ is a 2-uniform linear oriented hypergraph.

Another minor technical note: in order to create an oriented signed graph $(\Sigma,\beta)$ for a given signed graph $\Sigma$, many different orientations $\beta$ can be chosen so that Equation \ref{OSGrelation} is satisfied .  This freedom of choice does not usually cause any issues in the calculations.  There may, however, be specific situations when particular orientations are more favorable and this will be noted.

%==================================================================
%==================================================================
\subsection{Matrices and Oriented Hypergraphs}
%==================================================================
%==================================================================

Let $G=(H,\sigma)$ be an oriented hypergraph.  The \emph{adjacency matrix} $A(G)=(a_{ij})\in \mathbb{R}^{n\times n}$ is defined by
\begin{equation*}
a_{ij}=
\begin{cases}
\displaystyle\sum_{e\in E}\sgn_{e}(v_{i},v_{j}) &\text{if $v_i$ is adjacent to $v_j$},\\
0 &\text{otherwise.}
\end{cases}
\end{equation*}
If $v_i$ is adjacent to $v_j$, then
\begin{align*}
a_{ij}=\sum_{e\in E}\sgn_e(v_i,v_j)=\sum_{e\in E}\sgn_e(v_j,v_i)=a_{ji}.
\end{align*}
Therefore, $A(G)$ is symmetric.

The \emph{incidence matrix} $\mathrm{H}(G)=(\eta _{ij})$ is the $n\times m$ matrix, with entries in $\{\pm 1,0\}$, defined by 
\begin{equation*}
\eta _{ij}=
\begin{cases} \sigma(v_{i},e_{j}) & \text{if }(v_{i},e_{j})\in \mathcal{I},\\
0 &\text{otherwise.}
\end{cases}
\end{equation*}
As with hypergraphs, the incidence matrix provides a convenient relationship between an oriented hypergraph and its incidence dual.

\begin{lemma}[\cite{ReffRusnak1},Theorem 4.1]\label{OHIncidenceMatrixDualTranspose}
If $G$ is an oriented hypergraph, then $\mathrm{H}(G)^{\text{T}}=\mathrm{H}(G^{\ast })$.
\end{lemma}

The \emph{degree matrix} of an oriented hypergraph $G$ is defined as $D(G):=\text{diag}(d_1,d_2,\allowbreak\ldots,d_n)$.  The \emph{Laplacian matrix} is defined as $L(G):=D(G)-A(G).$  The Laplacian matrix of an oriented hypergraph can also be written in terms of the incidence matrix.

\begin{lemma}[\cite{ReffRusnak1}, Corollary 4.4]\label{OHLapIncidenceRelation}
If $G$ is an oriented hypergraph, then
\begin{enumerate}
\item $L(G)=D(G)-A(G)=\mathrm{H}(G)\mathrm{H}(G)^{\text{T}}$, and
\item $L(G^{\ast })=D(G^{\ast})-A(G^{\ast })=\mathrm{H}(G)^T\mathrm{H}(G).$
\end{enumerate}
\end{lemma}

Since the eigenvalues of a symmetric matrix $A\in\mathbb{R}^{n\times n}$ are real, we will assume that they are labeled and ordered according to the following convention:
\[ \lambda_{\min}(A)=\lambda_n(A) \leq \lambda_{n-1}(A) \leq \cdots \leq \lambda_2(A) \leq \lambda_1(A)=\lambda_{\max}(A).\]

%==================================================================
%==================================================================
\section{Intersection Graphs}\label{IntersectionGraphsSection}
%==================================================================
%==================================================================

The \emph{$k$-section} (\emph{clique graph}) of an oriented hypergraph $G=(V,E,\mathcal{I},\sigma)$ is the oriented hypergraph $[G]_k$ with the same vertex set as $G$ and edge set consisting of $f\subseteq V$ that satisfies either of the following:
\begin{enumerate}
\item $|f|=k$ and $f\subseteq e$ for some $e\in E$, or
\item $|f|<k$ and $f= e$ for some $e\in E$.
\end{enumerate}
The incidence signs for each $f$ are carried over from $e$ so that $\sigma_{[G]_k}(v,f)=\sigma_G(v,e)$.  The $k$-section of an oriented hypergraph is a generalization of the $k$-section of a hypergraph (see, for example, Berge \cite[p.26]{MR1013569}).

The \emph{strict $k$-section} is the oriented hypergraph $\llbracket G\rrbracket _k$ that is the same as $[G]_k$, but without condition (2).

The \emph{intersection graph} (\emph{line graph}, or \emph{representative graph}) of $G$ is the oriented hypergraph $\Lambda(G)$ whose vertices are the edges of $G$, and edges of the form $ef$ whenever $e\cap f \neq \varnothing$ in $G$.  The incidence signs are carried over to the intersection graph so that if $v\in e\cap f$ in $G$, then $\sigma_{\Lambda(G)}(e,ef)=\sigma(v,e)$.  The intersection graph of an oriented hypergraph simultaneously generalizes the intersection graph of a hypergraph \cite[p.31]{MR1013569} and the line graph of a signed graph introduced by Zaslavsky \cite{MITTOSSG}.  See Figure \ref{OHSectionGraphs} for an example of a 2-section, strict 2-section and intersection graph of an oriented hypergraph $G$.

\begin{figure}[!ht]
    \includegraphics[scale=0.7]{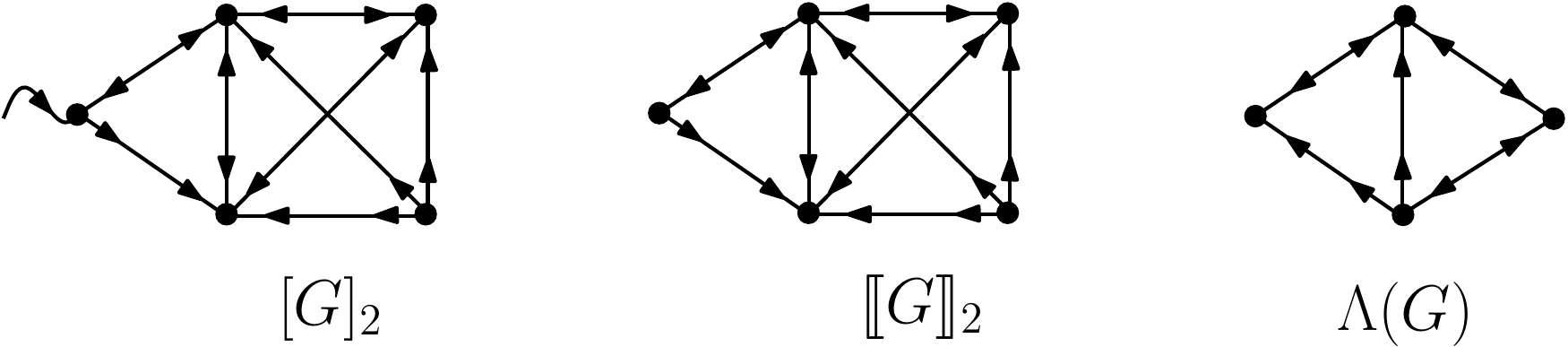}\centering
    \caption{The 2-section $[G]_2$, strict 2-section $\llbracket G \rrbracket_2$ and intersection graph $\Lambda(G)$ of $G$ from Figure \ref{OHEx}.}\label{OHSectionGraphs}
\end{figure}

The following theorem shows how to obtain the intersection graph of an oriented hypergraph $G$ by finding the strict 2-section of the dual $G^*$.  This generalizes Berge's known result for hypergraphs \cite[Prop. 1, p.33]{MR1013569}.

\begin{theorem}\label{LineGraphDualRelationship} If $G$ is a linear oriented hypergraph, then
$ \Lambda(G) = \llbracket G^*\rrbracket _2$.
\end{theorem}
\begin{proof}
The equivalences
\[v \in V(\Lambda(G))\leftrightarrow v\in E(G) \leftrightarrow v \in V(G^*) \leftrightarrow v\in V(\llbracket G^*\rrbracket _2),\]
verifies $V(\Lambda(G))=V(\llbracket G^*\rrbracket _2)$.  Similarly, the equivalences
\begin{align*} 
e_ie_j \in E(\Lambda(G)) &\leftrightarrow \exists e_i, e_j \in E(G) \text{ such that } \exists v\in V(G) \text{ such that } v\in e_i \cap e_j, \\
&\leftrightarrow \exists e_i, e_j \in V(G^*) \text{ such that } \exists v\in E(G^*) \text{ such that } e_i, e_j\in v,\\
&\leftrightarrow \exists e_i, e_j \in V(\llbracket G^*\rrbracket _2) \text{ such that } \exists v\in E(\llbracket G^*\rrbracket _2) \text{ such that } e_i, e_j\in v,\\
&\leftrightarrow e_i e_j \in E(\llbracket G^*\rrbracket _2),
\end{align*} 
confirms $E(\Lambda(G))=E(\llbracket G^*\rrbracket _2)$.  Finally, along with the inherited incidences, the equivalences
\[ \sigma_{\Lambda(G)}(e_i,e_ie_j) = \sigma_{G}(v,e_i)   = \sigma_{G^*}(e_i,v) =  \sigma_{\llbracket G^*\rrbracket _2}(e_i,e_ie_j),\]
show that the incidence signs are the same for $\Lambda(G)$ and $\llbracket G^*\rrbracket _2$.  Therefore, $\Lambda(G)=\llbracket G^*\rrbracket _2$.
\end{proof}

\begin{corollary} If $G$ is a linear oriented hypergraph, then $\Lambda(G^*) = \llbracket G\rrbracket _2$.
\end{corollary}

As a generalization of Berge's result, for any simple oriented signed graph $G$, there is a linear oriented hypergraph $H$ that has $G$ as its intersection graph.  It turns out that $H=G^*$ is one such oriented hypergraph.
\begin{corollary}\label{IntersectionGraphExistenceSG} If $G$ is a 2-uniform linear oriented hypergraph (that is, $G$ is an simple oriented signed graph), then $\Lambda(G^*)=G$.
\end{corollary}

Moreover, for any simple oriented signed graph $G$, there are infinitely many linear oriented hypergraphs with $G$ as their intersection graph.
\begin{corollary}\label{InfiniteFamilyIntersectionGraphExistenceSG}  If $G$ is a 2-uniform linear oriented hypergraph (that is, $G$ is an simple oriented signed graph), then there exists an infinite family of linear oriented hypergraphs $\mathcal{H}$ such that for any $H\in\mathcal{H}$,  $\Lambda(H)=G$.
\end{corollary}
\begin{proof}
Consider the dual $G^*$ and pick an edge $e\in E(G^*)$ that has nonzero size $i$.  Let $j\in\{i+1,i+2,\ldots\}$ and let $H_j$ be the oriented hypergraph that is identical to $G^*$ except  edge $e$ is has an additional $j-i$ new vertices of degree 1 incident to it.  These additional vertices also create $j-i$ new incidences, all of which can be given an orientation of $+1$ (this choice is arbitrary).  Since these new vertices do not create any new edges, or new edges incident to $e$, it must be that the intersection graphs of $G^*$ and $H_j$ are the same.  By Corollary \ref{IntersectionGraphExistenceSG}, $G=\Lambda(G^*)=\Lambda(H_j)$.  Therefore, the infinite family $\mathcal{H}:=\{ H_j : j\geq i\}$ satisfies the corollary.
\end{proof}

If $k\geq \Delta(G)$, the process of enlarging the edge sizes can be continued to obtain a $k$-uniform oriented hypergraph that has $G$ as its line graph.

\begin{theorem}\label{infiniteUNIFORMfamilyTHM} If $G$ is a 2-uniform linear oriented hypergraph (that is, $G$ is an simple oriented signed graph), then for all $k\geq \Delta(G)$, there exists a $k$-uniform linear oriented hypergraph $H_k$ with $\Lambda(H_k) = G$.
\end{theorem}
\begin{proof}Same construction as the previous proof, but enlarge all the edges to have size $k$.  Since $k\geq \Delta(G)$ it is guaranteed that all the edges can be enlarged to the desired size.
\end{proof}

To illustrate Corollaries \ref{IntersectionGraphExistenceSG} and \ref{InfiniteFamilyIntersectionGraphExistenceSG}, and Theorem \ref{infiniteUNIFORMfamilyTHM} see Figure \ref{OH2uniformDualandLG}.
\begin{figure}[!ht]
 \includegraphics[scale=0.8]{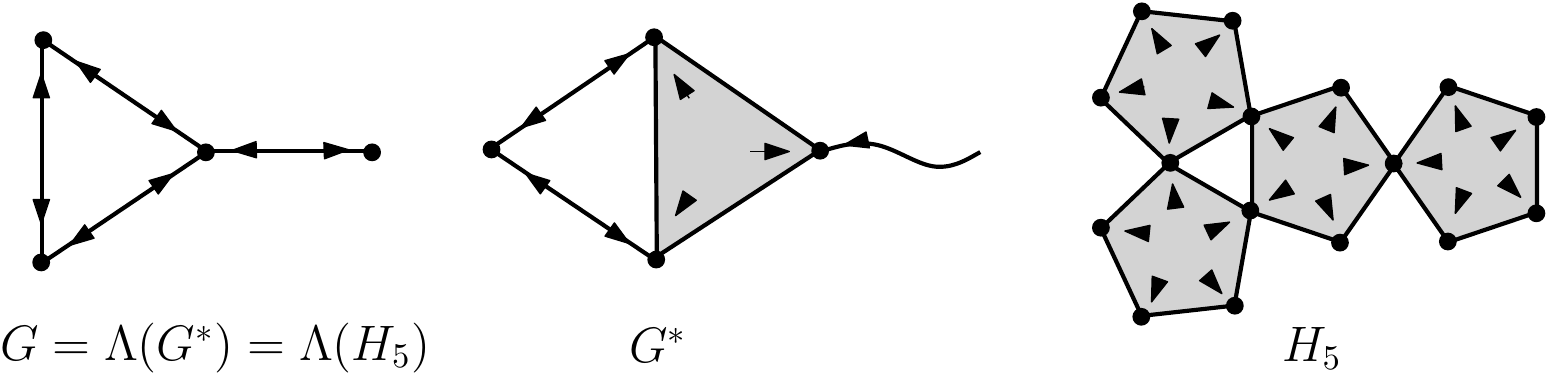}\centering
    \caption{A 2-uniform linear oriented hypergraph $G$, which is also the line graph of the dual $G^*$ and the 5-uniform oriented hypergraph $H_5$.}\label{OH2uniformDualandLG}
\end{figure}

Berge shows that every graph is the intersection graph of some linear hypergraph \cite[Prop. 2, p.34]{MR1013569} as is true for simple signed graphs.  In general, a signed graph $\Sigma$ is the intersection graph of infinitely many linear oriented hypergraphs. 
\begin{theorem}\label{SignedGraphIntGraphUniformThm}  If $\Sigma$ is a simple signed graph, then there is some linear oriented hypergraph $H$, with $\Lambda(H)=\Sigma$.
Moreover, for all $k\geq \Delta(\Sigma)$, there exists a $k$-uniform linear oriented hypergraph $H_k$ with $\Lambda(H_k)=\Sigma$.
\end{theorem}
\begin{proof} A simple signed graph $\Sigma$ has many possible orientations $\beta$ such that $G=(\Sigma,\beta)$ is a 2-uniform linear oriented hypergraph.  The result is immediate by Theorem \ref{infiniteUNIFORMfamilyTHM}.
\end{proof}

For a 2-regular oriented hypergraph $G$, the intersection graph $\Lambda(G)$ and the dual $G^*$ are identical.    Also, there are infinitely many $k$-uniform linear oriented hypergraphs whose intersection graphs are $G^*$.

\begin{corollary}\label{DualTwoRegular} If $G$ is a 2-regular linear oriented hypergraph, then $\Lambda(G)=G^*$. 
Moreover, for all $k\geq r(G)$, there exists a $k$-uniform oriented hypergraph $H_k$ such that
$\Lambda(H_k)=G^*$.
\end{corollary}
\begin{proof}
If $G$ is 2-regular, then $G^*$ is 2-uniform, hence the results are immediate by Corollaries \ref{IntersectionGraphExistenceSG} and \ref{infiniteUNIFORMfamilyTHM}.
\end{proof}

See Figure \ref{OH2regularDualandLG} for an example that illustrates Corollary \ref{DualTwoRegular}.
\begin{figure}[!ht]
 \includegraphics[scale=0.65]{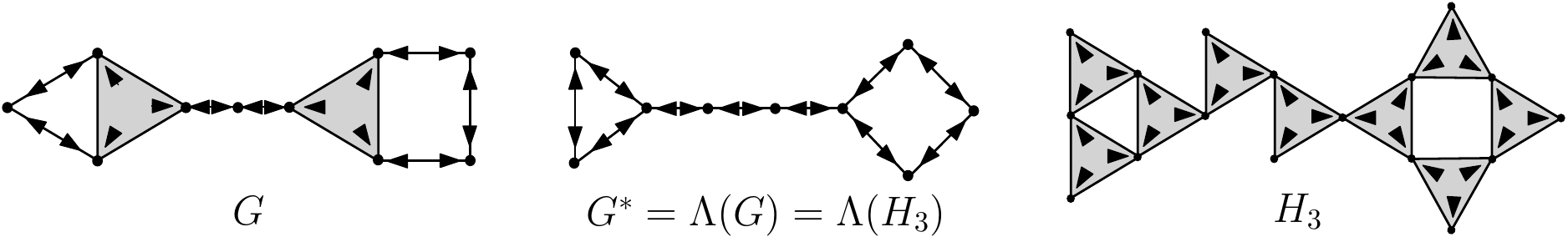}\centering
    \caption{A 2-regular linear oriented hypergraph $G$ and its dual $G^*$.  The dual $G^*$ is the line graph of both $G$ and the 3-uniform oriented hypergraph $H_3$.}\label{OH2regularDualandLG}
\end{figure}

%==================================================================
%==================================================================
\section{Matrices of Intersection Graphs}\label{MatricesofIGSection}
%==================================================================
%==================================================================

The strict 2-section $\llbracket G\rrbracket _2$ is essentially the oriented hypergraph created from the adjacencies in $G$, so on the level of adjacency matrices the oriented hypergraphs $G$ and $\llbracket G\rrbracket _2$ record the same information.
\begin{theorem}\label{OGAdjMatand2Section} If $G$ is an oriented hypergraph, then $A(G) = A(\llbracket G\rrbracket _2)=A([G]_2)$. 
\end{theorem}
\begin{proof}  By definition, $V(G)=V(\llbracket G\rrbracket _2)$, so both have adjacency matrices of the same size. If $v_i$ and $v_j$ are not adjacenct in $G$, then they are not adjacent in $\llbracket G\rrbracket _2$ and the $(i,j)$-entries of $A(G)$ and $A(\llbracket G\rrbracket _2)$ are 0.  Otherwise, $v_i$ and $v_j$ are adjacent and the $(i,j)$-entry of $A(G)$ is
\begin{align*}
\sum_{e\in E(G)} \sgn_e(v_i,v_j) &= \sum_{e\in E(G)} -\sigma_G(v_i,e) \sigma_G(v_j,e) \\
&= \sum_{f\in E(\llbracket G\rrbracket _2)} -\sigma_{\llbracket G\rrbracket _2}(v_i,f) \sigma_{\llbracket G\rrbracket _2}(v_j,f)\\
&= \sum_{f\in E(\llbracket G\rrbracket _2)} \sgn_f(v_i,v_j), 
\end{align*}
which is the $(i,j)$-entry of $A(\llbracket G\rrbracket _2)$.  This is also  the $(i,j)$-entry of $A([G]_2)$ since any extra edges of smaller size present in $[G]_2$ would not introduce any new adjacencies to consider. 
\end{proof}

This adjacency matrix relationship carries over to the dual in a significant way, showing that the dual and intersection graph have the same adjacency matrix.
\begin{corollary}\label{DualLineGAdjMat} If $G$ is a linear oriented hypergraph, then $A(G^*) = A(\Lambda(G))$.
\end{corollary}
\begin{proof} The result is an immediate consequence of Theorems  \ref{LineGraphDualRelationship} and \ref{OGAdjMatand2Section}.
\end{proof}

If $G$ is a $k$-uniform, we can specialize Lemma \ref{OHLapIncidenceRelation} as follows.
\begin{lemma}[\cite{ReffRusnak1},Corollary 4.5]\label{OHIncidenceMatrixDualTransposeREL}
If $G$ is a $k$-uniform oriented hypergraph, then 
\[L(G^*)=\mathrm{H}(G)^{\text{T}}\mathrm{H}(G)=kI-A(G^*).\]
\end{lemma}

\begin{theorem} Let $G$ be a $k$-uniform linear oriented hypergraph.  If $\lambda$ is an eigenvalue of $A(\Lambda(G))$ (or $A(G^*)$), then $\lambda \leq k$.
\end{theorem}
\begin{proof}
By Corollary \ref{DualLineGAdjMat} $A(\Lambda(G))=A(G^*)$, so these matrices can be used interchangeably.  Suppose that $\mathbf{x}$ is an eigenvector of $A(\Lambda(G))$ with associated eigenvalue $\lambda$.  By Lemma \ref{OHIncidenceMatrixDualTransposeREL} the following simplification can be made:
\[ L(G^*)\mathbf{x}=\mathrm{H}(G)^{\text{T}}\mathrm{H}(G)\mathbf{x}=\big(kI-A(\Lambda(G)\big)\mathbf{x} = (k-\lambda)\mathbf{x}.\]
Hence, $k-\lambda$ is an eigenvalue of  $L(G^*)=\mathrm{H}(G)^{\text{T}}\mathrm{H}(G)$.  Since $L(G^*)$ is positive semidefinite it must be that $k-\lambda \geq 0$ and therefore, $k\geq \lambda$.
\end{proof}

{\bf Question 1:} Suppose you are given a signed graph $\Sigma$ with all adjacency eigenvalues satisfying $\lambda \leq k$.  Does this mean $\Sigma$ is the intersection graph of some $k$-uniform oriented hypergraph?  We already know this is always the case if $\Delta(\Sigma) \leq k$.  For all signed graphs, it is known that all eigenvalues of $A(\Sigma)$ satisfy $\lambda \leq \Delta(\Sigma)$ \cite[Theorem 4.3]{MR2900705}.  So the question remains for situations when all adjacency eigenvalues satisfy $\lambda\leq k < \Delta(\Sigma)$ for some integer $k$.

\begin{figure}[!ht]
    \includegraphics[scale=0.7]{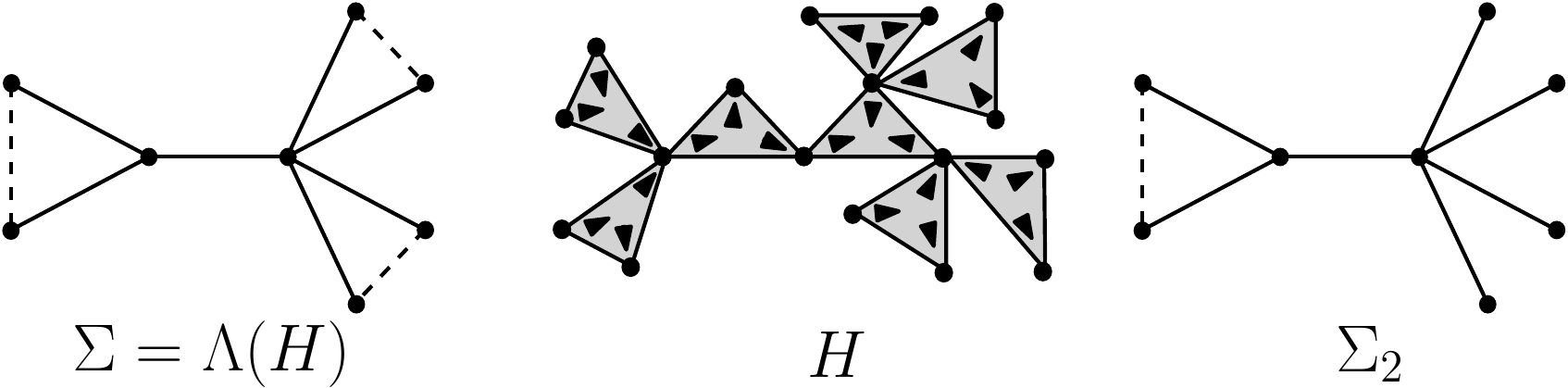}\centering
    \caption{A signed graph $\Sigma$ can be oriented so it is also the intersection graph of the depicted 3-uniform oriented hypergraph $H$.  A dashed edge denotes a sign of $-1$ and a solid edge denotes of a sign of $+1$.  The second signed graph $\Sigma_2$ is discussed in Example 1.}\label{SignedGraphEx}
\end{figure}

{\bf Example 1:} Consider the signed graph $\Sigma$ in Figure \ref{SignedGraphEx}.  In this example $\lambda_{\max}(A(\Sigma))\approx 2.03967 < k< \Delta(\Sigma)=5$.  So is $\Sigma$ the intersection graph of a 3-uniform or 4-uniform oriented hypergraph?  We already know that if $k\geq 5$ we can find a $k$-uniform oriented hypergraph $H_k$ with $\Lambda(H_k)=\Sigma$ by Theorem \ref{SignedGraphIntGraphUniformThm}.  However, it is possible to construct a 3-uniform oriented hypergraph $H$ with $\Sigma$ as its intersection graph, as shown in Figure \ref{SignedGraphEx}.  This construction can be modified to the 4-uniform case as well.  If, however, we consider the signed graph $\Sigma_2$ in Figure \ref{SignedGraphEx}, the situation is different.  For this signed graph, $\lambda_{\max}(A(\Sigma_2))\approx 2.31364 < k< \Delta(\Sigma_2)=5$, but it does not seem possible to have a 3 or 4-uniform oriented hypergraph with $\Sigma_2$ as its intersection graph.  Is there a nice classification of the exceptional cases?\\

If $G$ is $r$-regular, the dual relationship of Lemma \ref{OHLapIncidenceRelation} can also be simplified.

\begin{lemma}\label{RRegularMatrixRel}
If $G$ is an $r$-regular oriented hypergraph, then 
\[L(G)=\mathrm{H}(G)\mathrm{H}(G)^{\text{T}}=rI-A(G).\]
\end{lemma}

\begin{theorem} Let $G$ be a $r$-regular oriented hypergraph.  If $\lambda$ is an eigenvalue of $A(G)$, then $\lambda \leq r$.
\end{theorem}
\begin{proof}
Suppose that $\mathbf{x}$ is an eigenvector of $A(G)$ with associated eigenvalue $\lambda$.  By Lemma \ref{RRegularMatrixRel}
\[ L(G)\mathbf{x}=\mathrm{H}(G)\mathrm{H}(G)^{\text{T}}\mathbf{x}=\big(rI-A(G)\big)\mathbf{x} = (r-\lambda)\mathbf{x}.\]
Hence, $r-\lambda$ is an eigenvalue of the positive semidefinite matrix $L(G)$.  Therefore, it must be that $r\geq \lambda$.
\end{proof}

An oriented hypergraph and its incidence dual have the same nonzero Laplacian eigenvalues.  
\begin{lemma}[\cite{MR3218780},Corollary 4.2]\label{OHandDualLapEvals}
If $G$ is an oriented hypergraph, then $L(G)$ and $L(G^*)$ have the same nonzero eigenvalues.
\end{lemma}

{\bf Question 2:}  If $G$ is 2-regular, then $\Lambda(G)=G^*$ (see Corollary \ref{DualTwoRegular} above).  In this situation $L(G)$, $L(G^*)$ and $L(\Lambda(G))$ all have the same nonzero eigenvalues by Lemma \ref{OHandDualLapEvals}.  If $G$ is not 2-regular, when are the eigenvalues of $L(\Lambda(G))$ and $L(G^*)$ different?  By Corollary \ref{DualLineGAdjMat} the only difference between $L(\Lambda(G))$ and $L(G)$ in general is:
\[ L(G^*) - L(\Lambda(G)) = D(G^*) - D(\Lambda(G)).\]
Hence, a sufficient condition can be easily stated, yet a full classification using the structure of $G$ alone would be more interesting. 

\begin{proposition}\label{DegreeSumUniqueSpectra}
Let $G$ be an oriented hypergraph that is not 2-regular.
\begin{enumerate}
\item If $\sum_{i=1}^m d_i^{G^*} > \sum_{i=1}^m d_i^{\Lambda(G)}$, then $\exists j\in \{1,\ldots,m\}$, with\\  $\lambda_j(L(G^*))> \lambda_j\big(L(\Lambda(G))\big )$.
\item If $\sum_{i=1}^m d_i^{G^*} < \sum_{i=1}^m d_i^{\Lambda(G)}$, then $\exists j\in \{1,\ldots,m\}$, with\\  $\lambda_j(L(G^*))< \lambda_j\big(L(\Lambda(G))\big )$.
\end{enumerate}
\end{proposition}
\begin{proof} The trace of $L(G^*)$ is $\tr(L(G^*))=\sum_{i=1}^m d_i^{G^*}= \sum_{i=1}^m \lambda_i(L(G^*))$ and the trace of $L(\Lambda(G))$ is $\tr(L(\Lambda(G)))=\sum_{i=1}^m d_i^{\Lambda(G)}= \sum_{i=1}^m \lambda_i(L(\Lambda(G)))$.  Therefore, a strict increase in trace must result in a strict increase in at least one of the eigenvalues.
\end{proof}

{\bf Example 2:} Consider the oriented hypergraph $G$ in Figure \ref{OHEx}.  The dual $G^*$ and intersection graph $\Lambda(G)$ are depicted in Figures \ref{OHDUAL} and \ref{OHSectionGraphs}.  Condition (2) of Proposition \ref{DegreeSumUniqueSpectra} is met and the conclusion can be seen in the Laplacian eigenvalues approximated in Table \ref{LapEigenvaluTable1}.\\

{\bf Example 3:} Consider the oriented hypergraph $G_2$ in Figure \ref{OHDualandIntGraphwithsameLapE} together with its dual $G_2^*$ and intersection graph $\Lambda(G_2)$.  Condition (1) of Proposition \ref{DegreeSumUniqueSpectra} is met and the conclusion can be seen in the Laplacian eigenvalues approximated in Table \ref{LapEigenvaluTable1}.
%%%

\begin{table}[!ht]
\centering
    \begin{tabular}{ c| c| c||c| c|}
    \cline{2-5}
      & $L(\Lambda(G))$ &$ L(G^*)$ & $L(\Lambda(G_2))$ &$ L(G_2^*)$  \\
    \hline
    \multicolumn{1}{|c|}{$\lambda_1$} & 4.73205 & 4.56155 &2&4\\
    \hline
     \multicolumn{1}{|c|}{$\lambda_2$}  & 3.41421 & 3.73205&0&2\\
    \hline
    \multicolumn{1}{|c|}{$\lambda_3$}  & 1.26795 & 0.43845&\cellcolor{gray}&\cellcolor{gray}\\
    \hline
  \multicolumn{1}{|c|}{$\lambda_4$}  & 0.58579 & 0.26795&\cellcolor{gray}&\cellcolor{gray}\\
    \hline
\multicolumn{1}{|c|}{$\tr$}  & 10 & 9&2&6\\
    \hline
    \end{tabular}
\caption{Approximate Laplacian eigenvalues of $\Lambda(G)$ and $G^*$ from Figures \ref{OHDUAL} and \ref{OHSectionGraphs}, as well as $\Lambda(G_2)$ and $G_2^*$ in Figure \ref{OHDualandIntGraphwithsameLapE} .}\label{LapEigenvaluTable1}
    \end{table}

{\bf Example 4:} Consider the oriented hypergraph $G_3$ in Figure \ref{OHDualandIntGraphwithsameLapE}.  Here $G_3$ is not 2-regular, and yet the eigenvalues of $L(G_3^*)$ and $L(\Lambda(G_3))$ are the same.  In fact, $L(G_3^*)=L(\Lambda(G_3))$.   A full classification of when the Laplacian eigenvalues are the same for $G^*$ and $\Lambda(G)$ would be of interest.

\begin{figure}[!ht]
    \includegraphics[scale=0.7]{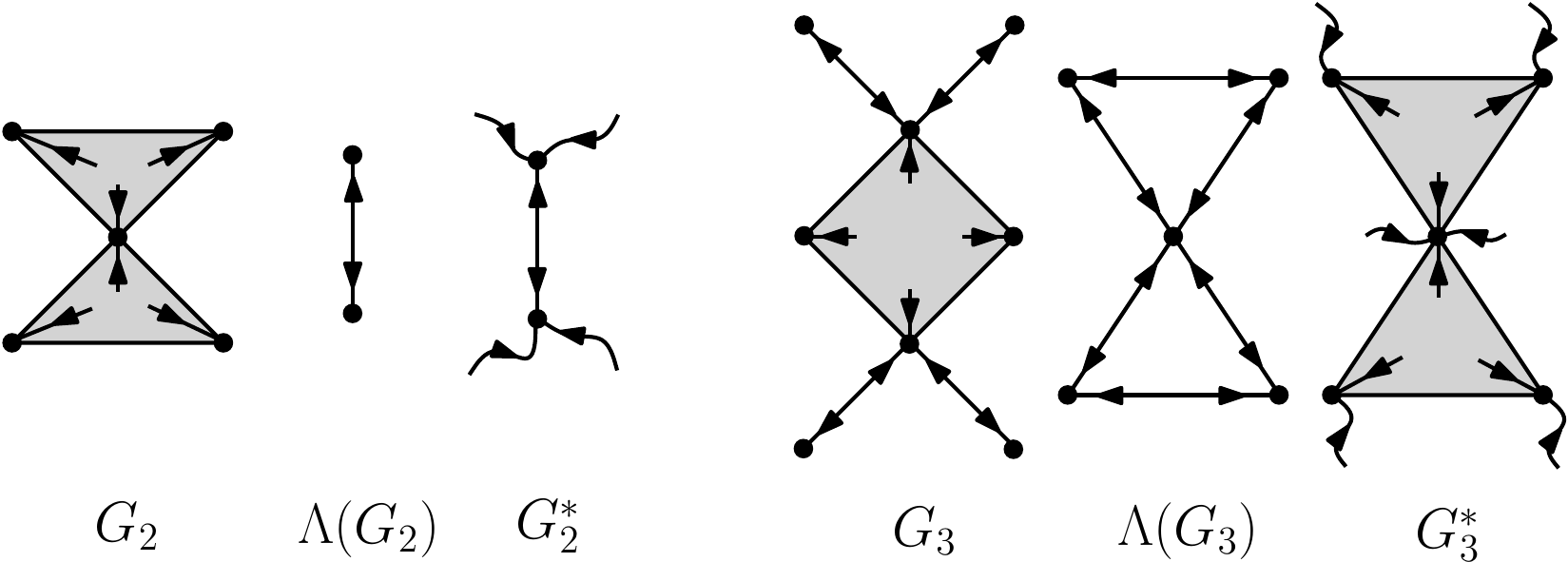}\centering
    \caption{Oriented hypergraphs considered in Examples 3 and 4.}\label{OHDualandIntGraphwithsameLapE}
\end{figure}

%%%%%%%%%%%%%%%%%%%%%%%%%%%%%%%%%%%%%%%
%%%%%%%%%%%%%%%%%%%%%%%%%%%%%%%%%%%%%%%
\section{Switching}\label{SwitchingSec}
%%%%%%%%%%%%%%%%%%%%%%%%%%%%%%%%%%%%%%%
%%%%%%%%%%%%%%%%%%%%%%%%%%%%%%%%%%%%%%%

A \emph{vertex-switching function} is any function $\zeta\colon V\rightarrow \{-1,+1\}$.  \emph{Vertex-switching} the oriented hypergraph $G=(H,\sigma)$ means replacing $\sigma$ with $\sigma^{\zeta}$, defined by
\begin{equation}
\sigma^{\zeta}(v,e)=\zeta(v)\sigma(v,e);
\end{equation}
producing the oriented hypergraph $G^{\zeta}=(H,\sigma^{\zeta})$.

An \emph{edge-switching function} is any function $\xi\colon E\rightarrow \{-1,+1\}$.  \emph{Edge-switching} the oriented hypergraph $G=(H,\sigma)$ means replacing $\sigma$ with $\sigma^\xi$, defined by 
\begin{equation}
\sigma^{\xi}(v,e)=\sigma(v,e)\xi(e);
\end{equation}
producing the oriented hypergraph $G^{\xi}=(H,\sigma^{\xi})$.  To make things more compact we will write $G^{(\zeta,\xi)}=(G,\sigma^{(\zeta,\xi)})$ when $G$ is both vertex-switched by $\zeta$ and edge-switched by $\xi$.  Switching changes the the adjacency signatures in the following way: 
\begin{align*}
\sgn_e^{(\zeta,\xi)}(v_i,v_j)=-\sigma^{(\zeta,\xi)}(v_i,e)\sigma^{(\zeta,\xi)}(v_j,e)&=-[\zeta(v_i)\sigma(v_i,e)\xi(e)][\zeta(v_j)\sigma(v_j,e)\xi(e)]\\
&=-\zeta(v_i)\sigma(v_i,e)\xi(e)^2\sigma(v_j,e)\zeta(v_j)\\
&=\zeta(v_i)\sgn_e(v_i,v_j)\zeta(v_j).
\end{align*}
The adjacency signatures are conjugated by the vertex-switching $\zeta$ and invariant under the edge-switching $\xi$. These switching operations can be encoded using matrices.

For a vertex-switching function $\zeta\colon V\rightarrow\{+1,-1\}$, we define the diagonal matrix $D_n(\zeta):=\text{diag}\big(\zeta(v_1),\zeta(v_2),\ldots,\zeta(v_n)\big)$.   Similarly for an edge-switching function $\xi\colon E\rightarrow \{+1,-1\}$, we define $D_m(\xi):=\text{diag}\big(\xi(e_1),\xi(e_2),\ldots,\xi(e_m)\big)$.  The following shows how to calculate the switched oriented hypergraph's incidence, adjacency and Laplacian matrices extending \cite[Propositions 3.1 and 4.3]{ReffRusnak1}.       

\begin{lemma}\label{OHLAHSwitchingSimilarityTrans} Let $G$ be an oriented hypergraph.  Let $\zeta\colon V\rightarrow\{+1,-1\}$ be a vertex-switching function on $G$, and  $\xi\colon E\rightarrow \{+1,-1\}$ be an edge-switching function on $G$.  Then 
\begin{enumerate}
\item $\Eta\big(G^{(\zeta,\xi)}\big)=D_n(\zeta)\Eta(G)D_m(\xi)$,
\item $A\big(G^{(\zeta,\xi)}\big)=D_n(\zeta) A(G) D_n(\zeta)$, and
\item $L\big(G^{(\zeta,\xi)}\big)=D_n(\zeta) L(G) D_n(\zeta)$.
\end{enumerate}
Moreover, 
\begin{enumerate}
\item[(4)] $\Eta\big((G^*)^{(\xi,\zeta)}\big)=D_m(\xi)\Eta(G^*)D_n(\zeta)$,
\item[(5)] $A\big((G^*)^{(\xi,\zeta)}\big)=D_m(\xi) A(G^*) D_m(\xi)$, and
\item[(6)] $L\big((G^*)^{(\xi,\zeta)}\big)=D_m(\xi) L(G^*) D_m(\xi)$.
\end{enumerate}
\end{lemma}

Since switching results in similarity transformations  for both the adjacency and Laplacian matrices, it also preserves the respective eigenvalues.  A specialized version of this situation appears in \cite[Lemmas 3.1 and 4.1]{MR3218780}.

\begin{theorem} Let $G$ be an oriented hypergraph.  Let $\zeta\colon V\rightarrow\{+1,-1\}$ be a vertex-switching function on $G$, and  $\xi\colon E\rightarrow \{+1,-1\}$ be an edge-switching function on $G$. Then
\begin{enumerate}
\item $A(G)$ and $A\big(G^{(\zeta,\xi)}\big)$ have the same eigenvalues.
\item $L(G)$ and $L\big(G^{(\zeta,\xi)}\big)$ have the same eigenvalues.
\end{enumerate}
Moreover,
\begin{enumerate}
\item[(3)] $A(G^*)$ and $A\big((G^*)^{(\xi,\zeta)}\big)$ have the same eigenvalues.
\item[(4)] $L(G^*)$ and $L\big((G^*)^{(\xi,\zeta)}\big)$ have the same eigenvalues.
\end{enumerate}
\end{theorem}

Lemma \ref{OHandDualLapEvals} can be generalized to find a large family of oriented hypergraphs which have the same nonzero Laplacian eigenvalues.

\begin{corollary}
Let $G$ be an oriented hypergraph.  Let $\zeta_1$ and $\zeta_2$ be vertex-switching functions on $G$, and let $\xi_1$ and $\xi_2$ be edge-switching functions on $G$. Then $L\big(G^{(\zeta_1,\xi_1)}\big)$ and $L\big((G^*)^{(\xi_2,\zeta_2)}\big)$ have the same nonzero eigenvalues.
\end{corollary}

If two oriented hypergraphs are the same up to a vertex or edge switching, then the corresponding duals, strict 2-sections and intersection graphs have related switching relationships.
\begin{theorem} Let $G_1=(H,\sigma_1)$ and $G_2=(H,\sigma_2)$ be linear oriented hypergraphs.  Let $\zeta$ be a vertex-switching function on $G_1$ and $G_2$, and  $\xi$ be an edge-switching function on $G_1$ and $G_2$. If $G_1=G_2^{(\zeta,\xi)}$, then
\begin{enumerate}
\item $ G_1^*=(G_2^*)^{(\xi,\zeta)}$. 
\end{enumerate}
Moreover, there exists edge switching functions $\hat{\xi}$ on $\llbracket G_2\rrbracket _2$ and $\hat{\zeta}$ on $\Lambda(G_2)$ such that
\begin{enumerate}
\item[(2)] $\llbracket G_1\rrbracket _2=\llbracket G_2\rrbracket _2^{(\zeta,\hat{\xi})}$, and
\item[(3)] $\Lambda(G_1)=\Lambda(G_2)^{(\xi,\hat{\zeta})}$. 
\end{enumerate}
\end{theorem}

\begin{proof}
Since $\zeta$ is a vertex-switching function on $G_1$ and $G_2$, by duality, $\zeta$ is an edge-switching function on $G_1^*$ and $G_2^*$.  Similarly $\xi$ becomes a vertex-switching function on $G_1^*$ and $G_2^*$. Now (1) follows from the incidence simplifications:
\begin{align*}
\sigma^{(\xi,\zeta)}_{G_2^*}(e,v)=\xi(e)\sigma_{G_2^*}(e,v)\zeta(v)=\zeta(v)\sigma_{G_2}(v,e)\xi(e)=\sigma^{(\zeta,\xi)}_{G_2}(v,e)&=\sigma_{G_1}(v,e)\\
&=\sigma_{G_1^*}(e,v).
\end{align*}
\smallskip
The strict 2-sections $\llbracket G_1\rrbracket _2$ and $\llbracket G_2\rrbracket _2$ both have the same vertex set as $H$ and therefore $\zeta$ is also a vertex-switching function for both oriented hypergraphs.  Define the edge-switching function $\hat{\xi}\colon E(\llbracket G_2\rrbracket _2)\rightarrow \{+1,-1\}$ as an extension of $\xi$ by the following rule: if $f\in E(\llbracket G_2\rrbracket _2)$ is derived from $e\in E(G_2)$ by definition (see the beginning of Section \ref{IntersectionGraphsSection}), then $\hat{\xi}(f)=\xi(e)$.  Now (2) follows from the following incidence calculation:
\begin{align*}
\sigma^{(\zeta,\hat{\xi})}_{\llbracket G_2\rrbracket _2}(v,f)=\zeta(v)\sigma_{\llbracket G_2\rrbracket _2}(v,f)\hat{\xi}(f)=\zeta(v)\sigma_{\llbracket G_2\rrbracket _2}(v,f)\xi(e)&=\zeta(v)\sigma_{G_2}(v,e)\xi(e)\\
&=\sigma^{(\zeta,\xi)}_{G_2}(v,e)\\
&=\sigma_{G_1}(v,e)\\
&=\sigma_{\llbracket G_1\rrbracket _2}(v,f).
\end{align*}
Finally, the result of (3) follows from (1), (2) and Theorem \ref{LineGraphDualRelationship}.
\end{proof}

%%%%%%%%%%%%%%%%%%%%%%%%%%%%%%%%%%%%%%%
%%%%%%%%%%%%%%%%%%%%%%%%%%%%%%%%%%%%%%%
\section{Balanced Incomplete Block Designs}\label{BIBDSec}
%%%%%%%%%%%%%%%%%%%%%%%%%%%%%%%%%%%%%%%
%%%%%%%%%%%%%%%%%%%%%%%%%%%%%%%%%%%%%%%

The matrix relationships found in the above sections produce a connection to balanced incomplete block designs.  There is a similar standard matrix relationship that appears in design theory which can be derived from a specialized oriented hypergraph.  The definitions and design theory results below are taken directly from \cite{MR2246267}.

Suppose $G=(H,+1)$ represents an oriented hypergraph $G$ with underlying hypergraph $H$ and all incidences labeled $+1$.

\begin{theorem}\label{OHBIBD} Suppose $G=(H,+1)$ is a $k$-uniform, $r$-regular oriented hypergraph where any two distinct vertices are adjacent exactly $\lambda$ times.
\[ L(G)=\Eta(G)\Eta(G)^{\text{T}}=(r-\lambda)I+\lambda J.\]
\end{theorem}
\begin{proof} Since $G$ is $r$-regular, $L(G) = \Eta(G)\Eta(G)^{\text{T}}=rI-A(G)$, by Lemma \ref{RRegularMatrixRel}.  Since all incidences are labeled $+1$, this forces all adjacency signs to be $-1$ by definition.  Therefore, since each distinct pair of vertices $v_i$ and $v_j$ are adjacent exactly $\lambda$ times, if $i\neq j$, then the $(i,j)$-entry of $A(G)$ is $\sum_{e\in E}\sgn_{e}(v_{i},v_{j}) = -\lambda$.  Otherwise, $a_{ii}=0$.  Hence, the result follows by further simplifying the initial equation: $rI-A(G)=rI-(\lambda I-\lambda J)$.
\end{proof}

A {\it balanced incomplete block design} (BIBD) is a pair $(V,\mathcal{B})$, where $V$ is a $v$-set (i.e., $v=|V|$) and $\mathcal{B}$ is a collection of $b$ $k$-subsets of $V$ (called {\it blocks}) such that each element of $V$ is contained in exactly $r$ blocks, and any 2-subset of $V$ is contained in exactly $\lambda$ blocks.  The numbers $v$, $b$, $r$, $k$ and $\lambda$ are called the {\it parameters} of the BIBD. 

The {\it incidence matrix} of a BIBD $(V,\mathcal{B})$ with parameters $v$, $b$, $r$, $k$ and $\lambda$ is a $v\times b$ matrix $C=(c_{ij})$, where $c_{ij}=1$ when the $i$th element of $V$ occurs in the $j$th block of $\mathcal{B}$, and $c_{ij}=0$, otherwise.

A BIBD $(V,\mathcal{B})$ can be thought of as a $k$-uniform, $r$-regular oriented hypergraph $G=(H,+1)$ with $V(H)=V$, $E(H)=\mathcal{B}$ and for any two distinct vertices $v_i$ and $v_j$ are adjacent exactly $\lambda$ times.  From the Theorem \ref{OHBIBD}, the same result known for BIBD can be established.

\begin{corollary} [\cite{MR2246267},Theorem 1.8]
Suppose $C$ is the incidence matrix of a balanced incomplete block design (BIBD) with parameters $v$, $b$, $r$, $k$, and $\lambda$.  Then
\[CC^{\text{T}}=(r-\lambda)I+\lambda J.\]  
\end{corollary}

%\bibliographystyle{elsart-num-sort}
%\bibliographystyle{abbrv}
%\bibliographystyle{amsplain2}
%\bibliography{mybib}

\begin{thebibliography}{10}
%\expandafter\ifx\csname url\endcsname\relax
%  \def\url#1{\texttt{#1}}\fi
%\expandafter\ifx\csname urlprefix\endcsname\relax\def\urlprefix{URL }\fi

\bibitem{MR2791608}
B.~D. Acharya, Signed intersection graphs, J. Discrete Math. Sci. Cryptogr.
  13~(6) (2010) 553--569.

\bibitem{MR1013569}
C.~Berge, Hypergraphs, vol.~45 of North-Holland Mathematical Library,
  North-Holland Publishing Co., Amsterdam, 1989, combinatorics of finite sets,
  Translated from the French.

\bibitem{LineDiHyp}
J.-C. Bermond, F.~Ergincan, M.~Syska, Line directed hypergraphs, in:
  D.~Naccache (ed.), Cryptography and Security: From Theory to Applications,
  vol. 6805 of Lecture Notes in Computer Science, Springer Berlin Heidelberg,
  2012, pp. 25--34.
%\newline\urlprefix\url{http://dx.doi.org/10.1007/978-3-642-28368-0_5}

\bibitem{MR1744046}
A.~Bj{\"o}rner, M.~Las~Vergnas, B.~Sturmfels, N.~White, G.~M. Ziegler, Oriented
  Matroids, vol.~46 of Encyclopedia of Mathematics and its Applications, 2nd
  ed., Cambridge University Press, Cambridge, 1999.

\bibitem{Chen2015442}
Vinciane Chen, Angeline Rao, Lucas~J. Rusnak, and Alex Yang, \emph{A
  characterization of oriented hypergraphic balance via signed weak walks},
  Linear Algebra and its Applications \textbf{485} (2015), 442 -- 453.

\bibitem{MR1235565}
F.~R.~K. Chung, The {L}aplacian of a hypergraph, in: Expanding graphs
  ({P}rinceton, {NJ}, 1992), vol.~10 of DIMACS Ser. Discrete Math. Theoret.
  Comput. Sci., Amer. Math. Soc., Providence, RI, 1993, pp. 21--36.

\bibitem{MR2246267}
C.~J. Colbourn, J.~H. Dinitz (eds.), Handbook of Combinatorial Designs,
  Discrete Mathematics and its Applications (Boca Raton), 2nd ed., Chapman \&
  Hall/CRC, Boca Raton, FL, 2007.

\bibitem{MR2900714}
J.~Cooper, A.~Dutle, Spectra of uniform hypergraphs, Linear Algebra Appl.
  436~(9) (2012) 3268--3292.

\bibitem{MR1405722}
K.~Feng, W.-C.~W. Li, Spectra of hypergraphs and applications, J. Number Theory
  60~(1) (1996) 1--22.

\bibitem{MR1881264}
D.~Ferrero, C.~Padr{\'o}, Connectivity and fault-tolerance of hyperdigraphs,
  Discrete Appl. Math. 117~(1-3) (2002) 15--26.

\bibitem{MR1889911}
D.~Ferrero, C.~Padr{\'o}, Partial line directed hypergraphs, Networks 39~(2)
  (2002) 61--67.

\bibitem{MR1178183}
M.~A. Fiol, A.~S. Llad{\'o}, The partial line digraph technique in the design
  of large interconnection networks, IEEE Trans. Comput. 41~(7) (1992)
  848--857.

\bibitem{MR1325271}
J.~Friedman, A.~Wigderson, On the second eigenvalue of hypergraphs,
  Combinatorica 15~(1) (1995) 43--65.

\bibitem{MR1217096}
G.~Gallo, G.~Longo, S.~Pallottino, S.~Nguyen, Directed hypergraphs and
  applications, Discrete Appl. Math. 42~(2-3) (1993) 177--201, combinatorial
  structures and algorithms.


\bibitem{zbMATH03604927}
C.~{Greene}, {Acyclic orientations.}, {Higher Comb., Proc. NATO Adv. Study
  Inst., Berlin (West) 1976, 65-68 (1977).} (1977).

\bibitem{MR0130839}
F.~Harary, R.~Z. Norman, Some properties of line digraphs, Rend. Circ. Mat.
  Palermo (2) 9 (1960) 161--168.

\bibitem{MR3270186}
P.~Keevash, J.~Lenz, D.~Mubayi, Spectral extremal problems for hypergraphs,
  SIAM J. Discrete Math. 28~(4) (2014) 1838--1854.

\bibitem{MR2842309}
L.~Lu, X.~Peng, High-ordered random walks and generalized {L}aplacians on
  hypergraphs, in: Algorithms and models for the web graph, vol. 6732 of
  Lecture Notes in Comput. Sci., Springer, Heidelberg, 2011, pp. 14--25.

\bibitem{nikiforov}
V.~Nikiforov, An analytic theory of extremal hypergraph problems, preprint:
  \href{http://arxiv.org/pdf/1305.1073v2.pdf}{arXiv:1305.1073}.

\bibitem{MR2900705}
N.~Reff, Spectral properties of complex unit gain graphs, Linear Algebra Appl.
  436~(9) (2012) 3165--3176.
%\newline\urlprefix\url{http://dx.doi.org.ezproxy2.drake.brockport.edu/10.1016/j.laa.2011.10.021}

\bibitem{MR3218780}
N.~Reff, Spectral properties of oriented hypergraphs, Electron. J. Linear
  Algebra 27 (2014) 373--391.

\bibitem{ReffRusnak1}
N.~Reff, L.~J. Rusnak, An oriented hypergraphic approach to algebraic graph
  theory, Linear Algebra Appl. 437~(9) (2012) 2262--2270.
%\newline\urlprefix\url{http://dx.doi.org/10.1016/j.laa.2012.06.011}

\bibitem{MR1890984}
J.~A. Rodr{\'{\i}}guez, On the {L}aplacian eigenvalues and metric parameters of
  hypergraphs, Linear Multilinear Algebra 50~(1) (2002) 1--14.

\bibitem{MR3118956}
L.~J. Rusnak, Oriented hypergraphs: introduction and balance, Electron. J.
  Combin. 20~(3) (2013) Paper 48, 29.

\bibitem{Shi1992533}
C.-J. Shi, A signed hypergraph model of the constrained via minimization
  problem, Microelectronics Journal 23~(7) (1992) 533 -- 542.


\bibitem{MR1666019}
C.-J. Shi, J.~A. Brzozowski, A characterization of signed hypergraphs and its
  applications to {VLSI} via minimization and logic synthesis, Discrete Appl.
  Math. 90~(1-3) (1999) 223--243.

\bibitem{MR1120422}
T.~Zaslavsky, Orientation of signed graphs, European J. Combin. 12~(4) (1991)
  361--375.

\bibitem{MITTOSSG}
T.~Zaslavsky, Matrices in the theory of signed simple graphs, in: Advances in
  discrete mathematics and applications: {M}ysore, 2008, vol.~13 of Ramanujan
  Math. Soc. Lect. Notes Ser., Ramanujan Math. Soc., Mysore, 2010, pp.
  207--229.

\end{thebibliography}

\end{document}